\documentclass[12pt]{article}
\usepackage[dvips]{graphicx}
\usepackage{amssymb}
\usepackage{amsmath}

\newtheorem{theorem}{Theorem}
\newtheorem{proposition}{Proposition}

\newtheorem{lemma}{Lemma}
\newtheorem{example}{Example}

\newenvironment{proof}
{\begin{rm}\par\smallskip\noindent{\bf Proof.}\quad}{\QED\end{rm}}
\def\BBox{\rule{2mm}{3mm}}
\def\QED{\hfill$\BBox$}

\def\BBox{\rule{2mm}{3mm}}
\def\QED{\hfill$\BBox$}

\title{Efficient Computation of a Canonical Form for a Generalized P-matrix}

\author{Walter D. Morris, Jr.
}

\date{}

\begin{document}

\maketitle
\renewcommand{\thefootnote}{\fnsymbol{footnote}}

\begin{abstract}
We use recent results on algorithms for Markov decision problems to show that a canonical form
for a generalized P-matrix can be computed, in some important cases, by a strongly polynomial algorithm.
\end{abstract}
\section{Introduction}Suppose that the $m \times n$ real matrix $A$ has the block form $$ A = \left[ {\begin{array}{*{20}c}
 A_1& | & A_2 & | \cdots & | & A_m  \end{array} } \right], $$ with each column of block $A_j$ being of the form $e_j - \gamma p_{jk}$, with $e_j$ the $j^{th}$ standard basis vector and $p_{jk}$ a vector of nonnegative entries that sum to one. A {\it Markov decision problem}, as described in \cite{Ye}, is a linear program of the form:
maximize $v^Tb$, subject to $v^TA\le c^T$ where $b\in \mathbb{R}^m$ is positive and $c \in \mathbb{R}^n$.  Two recent papers, \cite{Ye} and subsequently \cite{HMZ}, advanced the theory of such problems significantly.  Ye proved that the simplex method, with Dantzig's rule for the entering variable, is a strongly polynomial algorithm to solve the Markov decision problem if the discount factor $\gamma$ is fixed. Such a positive result for complexity of a pivoting algorithm is rare. The paper of Hansen, Miltersen and Zwick showed a similar result for a considerable generalization of the Markov decision problem, in which one seeks a solution to the optimality conditions of the problem, with some inequalities reversed.

It is natural to seek applications of the advances achieved in these recent papers.  The sequence of papers by Kitihara and Mizuno, of which \cite{KM} is one example, generalizes Ye's results to other linear programs.

The matrix $A$ from a Markov decision problem has the {\it P-property}, which means that all square submatrices of $A$ formed by taking one column from each block have determinants of the same non-zero sign.  We consider the problem of
satisfying the optimality conditions of the Markov decision problem for more general $A$ with the P-property.  This problem is known as the generalized LCP (see \cite{CD}, \cite{GMR}).  We introduce this problem in Section 2.  Theorem 1 of that section describes a canonical form for $A$ with the P-property, introduced in \cite{WJ} to describe the set of ``n-step vectors" of an LCP.  One of our goals is to emphasize the significance of this form.  Theorem 2 describes a subclass of matrices with the P-property for which the analyses of \cite{Ye} can be directly applied.  This description is in terms of the canonical form of Theorem 1.  Theorems 1 and 2 are not essentially new, but the proofs supplied are more streamlined than earlier proofs.

Section 3 discusses the Markov decision problem.  The complexity results of \cite{Ye} and \cite{HMZ} have the term
$1-\gamma$ in the denominator.  This leads us to look for equivalent Markov decision problems for which this term is as large as possible.  A linear program to find such an optimal problem is formulated.  We call this linear program $LP(A)$.

Section 4 presents a two step algorithm to solve $LP(A)$.  This method finds the canonical form of Theorem 1 along the way.  Section 5 shows that the two step algorithm to solve $LP(A)$ is a strongly polynomial time algorithm, if the optimal value of $LP(A)$ is a fixed positive number.

\section{The P-property}
Suppose that the matrix $A \in \mathbb{R}^{m \times n}$ has its column set indexed by the
set of pairs $$\{(j,k): j=1,2,\ldots,m, k = 1,2,\ldots,n_j\},$$ where $n_1+n_2+\cdots+n_m = n$.
A {\it representative} submatrix of $A$ is an $m \times m$ submatrix for which, for each $j = 1,2,\ldots,m$ the $j^{th}$ column is indexed by $(j,k)$ for some $k$.

The matrix $A$ has the {\it P-property} if the determinants of all of its representative submatrices have the same nonzero sign.  (See Example 1 on Page 5 of this article.)
 Cottle and Dantzig \cite{CD}, Theorem 3, proved the existence part of the following. (Uniqueness was proved by Szanc \cite{S}.)
 \begin{proposition}  Suppose that $A$ has the P-property.  If $c \in \mathbb{R}^n$ then there exists a unique vector $v \in \mathbb{R}^m $ so that $c^T- v^T A \ge 0$, and
 for each $j=1,2,\ldots,m$, $(c^T-v^T A)_{jk} = 0$ for at least one $k \in \{1,2,\ldots,n_j\}$.
\end{proposition}
 Let $\hat{C}$ be the representative matrix of $A$ for
which column $j$ is column $(j,n_j)$ of $A$, for $j = 1,2,\ldots,m$. Form the matrix $\hat{A}$ from the columns $(j,k)$ of $\hat{C}^{-1}A$ for which $k < n_j$. Let $c_N$ be the subvector of $c$ with subscripts $(j,k)$ only for $k < n_j$. Then the vector $z = \hat{C}^Tv$ satisfies
$z \ge 0$ and $w= c_N - z^T \hat A \ge 0$, $z_j\prod_{k=1}^{n_j-1}w_{jk} =0$ for $j = 1,2,\ldots,m$.  Thus $z$ is a solution to the generalized linear complementarity problem defined by $\hat{A}^T$ and $c$.  If $n_j = 2$ for all $j$, then this generalized linear complementarity problem is a standard linear complementarity problem.

 A reformulation of the P-property that follows easily from \cite{CD}, Theorem 5, is that for every nonzero vector $x$ in the row space of $A$, there
 is a $j \in \{1,2,\ldots,m\}$ so that the coordinates $x_{jk}$ have the same nonzero sign for $k = 1,2,\ldots, n_j$.  We call this the sign-preserving property of the vector $x$.  Finally, Theorem 6 of \cite{CD} says that if $A$ has the P-property, then the row space of $A$ contains a positive vector.

 A polytope ${\cal P}_{A,b}$ defined by the system $Ax=b$, $x \ge 0$ is said to be combinatorially equivalent to a product of simplices if the solution to $Cx=b$ is positive for every representative submatrix $C$ of $A$.
In that case, we say that the representative submatrices of $A$ are the nondegenerate basic feasible solutions to $Ax=b$, $x \ge 0$.  If $n_j=2$ for $j=1,2,\ldots,m$, then such a polytope ${\cal P}_{A,b}$ is combinatorially equivalent to a cube.
If ${\cal P}_{A,b}$ is combinatorially equivalent to a product of simplices, then $A$ has the P-property.

 A square matrix with nonnegative off-diagonal entries is called a {\it Z-matrix.}  The following classical
 theorem \cite{FP} is central for Z-matrices.
 \begin{proposition} Let $M$ be a Z-matrix.  The following are equivalent.
 \begin{enumerate}
 \item $M$ is a P-matrix, i.e. all principal minors of $M$ are positive.
 \item The system $x>0, Mx>0$ has a solution.
 \item The inverse of $M$ exists and is nonnegative.
 \end{enumerate}
 \end{proposition}

A Z-matrix $M$ satisfying the conditions of Proposition 2 is called a K-matrix.

\begin{theorem} Suppose $A$ has the P-property.  Then there exists an $m \times m$ matrix $\overline{X}$ so that
\begin{enumerate}
\item $(\overline{X}A)_{ijk}>0$ whenever $i=j$.
\item $(\overline{X}A)_{ijk}<=0$ whenever $i \neq j$.
\item For every pair $(i,j)$ with $i \neq j$, there exists $k \in \{1,2,\ldots,n_j\}$ so that $(\overline{X}A)_{ijk}=0$.
\end{enumerate}

$\overline{X}$ is unique up to positive scaling of its rows.
\end{theorem}

\begin{proof} We recall the matrix $\hat A$ = $\hat{C}^{-1}A$, where $\hat{C}$ is the representative matrix of $A$ for
which column $j$ is column $(j,n_j)$ of $A$, for $j = 1,2,\ldots,m$.
Let $\hat A^m$ be the matrix obtained from $\hat A$ by deleting row $m$ and deleting columns $(m,1),(m,2),\ldots,(m,n_m)$.  Also, define
$c$ to be $-1$ times the vector obtained from row $m$ of $\hat A$ by deleting entries $(m,1),(m,2),\ldots,(m,n_m)$.  Then $\hat A^m$ has the P-property, so there is a vector $v \in \mathbb{R}^{m-1}$ satisfying $c^T-v^T \hat A^m \ge 0$, and for every $j \neq n$ there is a $k \in \{1,2,\ldots,n_j\}$ so that $(c^T-v^T \hat A^m)_{jk}=0.$  Appending a $1$ to the end of vector $v$, we get $(v,1)^T \hat{A}$ which has the negative of the sign pattern of $c^T-v^T \hat A^m$ on coordinates
not indexed by $(m,1),(m,2),\ldots,(m,n_m)$, and which is 1 on component $(m,n_m)$. The sign-preserving property of the row space of $A$ then implies that $(v,1)^T \hat{A}$ is positive on all components indexed by  $(m,1),(m,2),\ldots,(m,n_m)$.  Thus $(v,1)^T \hat{C}^{-1}$ can serve as row $m$ of the matrix $\overline{X}$.  A similar construction can be carried out for each of the rows of $\overline{X}$.
\end{proof}

Theorem 1 was proved in \cite{WJ} for the case $|n_j|=2$ for all $j$.  The algorithm in the Proof to construct $\overline{X}$ uses Proposition 1 applied to matrices with $m-1$ rows. This improves slightly upon the original proof
which applied Proposition 1 to matrices with $m$ rows.  Unfortunately, the complexity of finding the vector guaranteed by Proposition 1 is unknown for general $A$ with the P-property.
We will call the matrix $\overline{A} = \overline{X}A$ given by the previous theorem the {\it complementary Z-form} of $A$.

\begin{theorem}
Let $A$ have the P-property.  The following are equivalent:
\begin{enumerate}
\item There exists a positive vector $p\in\mathbb{R}^m$ so that $p^T\overline{A} >0$.
\item There exists a matrix $X \in \mathbb{R}^{m\times m}$ and a positive vector $p \in \mathbb{R}^m$ so that
$XA$ satisfies conditions (1) and (2) of Theorem 1, and $p^TXA>0$.
\item There exists a vector $b \in \mathbb{R}^m$ so that  for every representative submatrix $C$ of $A$ the solution to $Cx=b$ is positive.
\end{enumerate}
\end{theorem}

\begin{proof}  Clearly, (1) implies (2). If (2) holds, then every representative submatrix of $XA$ is a K-matrix.  Let $q$ be any positive vector.  Then for any representative submatrix $C$ of $XA$, the solution to $Cx=q$ is positive.  It follows that for every representative submatrix $D=\overline{X}^{-1}C$ of $A$, the solution to $Dx=b$ is positive for $b=\overline{X}^{-1}q$.  Thus (3) holds.
To show that (3) implies (1), suppose that the nondegenerate feasible bases of $Ax=b$, $x\ge 0$ are the representative submatrices of $A$.  Let $C$ be a representative submatrix of $\overline{X}A$ for which $C_{mj}= 0$, $j = 1,2,\ldots,m-1.$  The solution to $Cx=\overline{X}b$ must be positive, and $C_{mm}$ is positive, so component $m$ of $\overline{X}b$ is positive.  In a similar way, we see that all components of $\overline{X}b$ are positive.  It follows that every representative matrix of $\overline{A}$ is a K-matrix.  We add a column $(j,n_j+1)$, equal to the $j^{th}$ standard basis vector of $\mathbb{R}^m$, to $\overline{A}$ for each $j=1,2,\ldots,m$.  The resulting matrix retains the P-property, by property (1) of Proposition 2 and so it has a positive vector in its row space.  There is therefore a positive vector $p$ so that $p^T\overline{A} > 0$.
\end{proof}

The equivalence of (2) and (3) in Theorem 2 was already discussed in \cite{GMR}.  The equivalence of (1) to the other
conditions was discussed for the case $|n_j|=2 $ for all $j$ in \cite{WJ} and $\cite{CPS}$, section 4.8.  The proof above seems to be particularly straightforward.
The matrix $\overline{X}$ is useful for several reasons.  It provides a quick certificate for proving that the equivalent
conditions of Theorem 2 do {\it not} hold.  One needs only to display $\overline{A}$ and a representative submatrix $C$ of $\overline{A}$ for which the system $Cx\le0, x\ge 0, x\neq 0$ has a solution.  A different reason
arises if one wants to describe the set of vectors $b$ for which the representative submatrices of $A$ are the nondegenerate feasible bases of $Ax=b, x \ge 0$.  This set is $\{\overline{X}^{-1}q: q>0\}$, a simplicial cone.
In the case where $A=(I,M)$ where column $j$ of $I$ is indexed by $(j,1)$ and column $j$ of a square matrix $M$ is indexed by $(j,2)$ for all $j$, a vector $b$ for which the (\cite{CPS} section 4.8) representative submatrices of $A$ are the nondegenerate feasible bases of $Ax=b, x\ge 0$ is a so-called n-step vector for the matrix $M$ and can be used to solve any linear complementarity problem with matrix $M$ in at most $m$ pivots.  The location of the zeroes in $\overline{A}$ may offer useful combinatorial information.  If $|n_j|=2$ for all $j$, Theorem 2 holds, and one of the representative submatrices of $A$ is a diagonal matrix, then the paper \cite{FFGL} shows that the simplex method finds the vector guaranteed by Proposition 1 very quickly.

It follows that the efficient computation of $\overline{X}$ is of interest.

\section{Discounted Markov Decision Problems}

A discounted Markov decision problem is defined by an $m \times n$ matrix $A$, with the columns indexed by pairs $(j,k)$ as in the previous section, a vector $c \in \mathbb{R}^n$, and a scalar $\gamma$ between 0 and 1.  Every representative submatrix of $A$ is of the form $I-\gamma P$, with $I$ the $m \times m$ identity matrix and $P$ a nonnegative matrix with column sums equal to 1. The goal is to find a vector $v \in \mathbb{R}^m$ so that   $c^T-v^T A \ge 0$, and
 for each $j=1,2,\ldots,m$, $(c^T-v^T A)_{jk} = 0$ for at least one $k \in \{1,2,\ldots,n_j\}$.

It is easy to see that the matrix $A$ for a discounted Markov decision problem satisfies the P-property.  Moreover, it satisfies condition (2) of Theorem 2 with $X = I$ and $p = e^m$, the vector of length $m$ with all entries equal to one.

\begin{proposition} Suppose that the matrix $A$ has the P-property and satisfies the equivalent conditions of Theorem 2. Then the problem of finding a $v$ as in Proposition 1 is equivalent to solving a Markov decision problem.
\end{proposition}
\begin{proof}
Suppose that $A$ satisfies the equivalent conditions of Theorem 2.  Then there exists $X \in \mathbb{R}^{m \times m}$ satisfying (2) of Theorem 2, and by scaling the rows of $X$ we can assure that $(e^m)^TXA > 0$.  It will not generally be true that the entries of $(e^m)^TXA$ will all be equal, but we can enforce this by adding a row and column in a way that will not affect the application of any pivoting algorithm applied to the matrix $XA$.  Let $d = min_{(j,k)}\sum_i(XA)_{ijk}$.  We can augment $XA$ to an $(m+1) \times (n+1)$ matrix $\widetilde{XA}$ by first adding an $(n+1)^{st}$ column, indexed by $(m+1,1)$, of zeroes, and then adding a row of length $n+1$ for which entry $(j,k)$ is $d-\sum_{i=1}^m(XA)_{ijk}$ for $j=1,2,\ldots,m+1$, $k = 1,2,\ldots, n_j.$

Every representative submatrix of the matrix $\widetilde{XA}$ will contain the new column $(m+1,1)$.  Suppose that
$c$ is a vector in $\mathbb{R}^n$, and let $\tilde{c}$ be $c$ with a 0 appended as an $(m+1,1)$ component.
Suppose that $C$ is a representative submatrix of $A$ with column indices in a set $J \subseteq [n]$ and let $\tilde{C}$ be the representative submatrix of $\widetilde{XA}$ with columns indexed by $J \cup \{(m+1,1)\}$. Let $c_J$ be the subvector of $c$ with components indexed by $J$, and let $\tilde{c_J}$ be the subvector of $\tilde{c}$ with components indexed by $J \cup \{(m+1,1)\}$.  It is easy to see that a vector $p \in \mathbb{R}^m$ satisfies $v^T\tilde{C}=\tilde{c_J}$ if and only if $v^T_{m+1}=0$ and the vector $v'$ obtained from $v$ by deleting the last component, satisfies
$(v')^TXC=c_J$.  Thus, any pivoting algorithm applied to solving the linear program: minimize $c^Tx$ subject to
$Ax=b, x \ge 0$, where the representative matrices of $A$ are the nondegenerate feasible bases of $Ax=b, x \ge 0$, will go through the same steps as one applied to solving the Markov decision problem given by $\widetilde{XA}$ and $\tilde{c}$.
\end{proof}

One example of a linear program that Proposition 2 shows is equivalent to a Markov decision problem is the Klee-Minty
cube, studied in \cite{KM}.   The transformation of the Klee-Minty cube into a Markov decision problem necessarily leads to a very small value of $1-\gamma$.
Two recent papers give polynomial time complexity results on pivoting algorithms for the discounted Markov decision problem.  The paper \cite{Ye} proves that Dantzig's rule of choosing the entering variable with the smallest reduced cost, applied to a discounted Markov decision problem with an $m \times n$ matrix and discount factor $\gamma$, takes at most $\frac{m(n-m)}{1-\gamma}\cdot log(\frac{m^2}{1-\gamma})$ pivots to find the solution.  The paper \cite{HMZ} proves that the strategy iteration method requires at most $O(\frac{n}{1-\gamma}log(\frac{m}{1-\gamma})$ iterations.  These bounds are comparable because an iteration of the strategy iteration method requires more work than an iteration of the simplex method.  Both of these bounds are polynomials in $n$ and $m$ if $\gamma$ is not assumed to be part of the input data.  One should note that the paper \cite{HMZ} solves a considerable generalization of the Markov decision problem
with the same complexity.  This generalization allows for a partition of the set $[m]$ into two parts $M_1$ and $M_2$, and the vector  $v$ must satisfy $(c^T-v^T A)_{jk} \ge 0$ for $j \in M_1$ and $(c^T-v^TA)_{jk} \le 0$ for $j \in M_2$, in addition to the complementarity condition which says that for each $j \in [m]$, $(c^T-v^T A)_{jk} = 0$ for some $k \in \{1,2,\ldots,n_j\}$.

Given a matrix $A$ with the P-property, it seems desirable to look for an equivalent Markov decision problem matrix  with $1-\gamma$ as large as possible.  This leads to the following linear program:$$  {\begin{array}{*{20}c}
 \text{maximize } d \text{ subject to}  \\
    (XA)_{ijk} \le 1 \text{ whenever } i=j \\
    (XA)_{ijk} \le 0 \text{ whenever } i \neq j \\
    (e^m)^TXA \ge d(e^n)^T
      \end{array} }.  $$

This linear program is similar to one from \cite{P}.

We will call this linear program $LP(A)$.  It is clear that $LP(A)$ always has the feasible solution $X = 0$, $d=0$, and that $LP(A)$ has a solution with positive optimal value if and only if one of the equivalent conditions of Theorem 2 is satisfied.

\begin{example}
Let
$$A = \left[ {\begin{array}{*{20}c}
 4 & 4& | & -1 & -3 & | & -2 & -1 \\
    -2 & -1 & |& 4 & 4 & | &-1& -1   \\
    -1 & -2 & |& -1 & 0 & | & 4& 4
      \end{array} } \right]. $$

      Using Maple 11 to solve $LP(A)$ yields the solution
      $$X = \left[ {\begin{array}{*{20}c}
 111/350 & 37/350 & 37/350 \\
 1/25 & 7/25& 2/25  \\
 4/35 & 3/35 & 2/7
      \end{array} } \right], d=33/70, $$ and
      $$XA = \left[ {\begin{array}{*{20}c}
 333/350 & 333/350& | & 0 & -37/70 & | & -111/350 & 0 \\
  -12/25 & -7/25 & |& 1 & 1 & | & -1/25 & 0 \\
  0 & -1/5 & |& -2/35 & 0 & | & 29/35 & 33/35
      \end{array} } \right]. $$

\end{example}

In this example, the columns appear in the order $(1,1),(1,2),(2,1),(2,2),(3,1),(3,2)$, and $n_1=n_2=n_3=2$.
Note that the matrix $A$ already satisfies the sign conditions (1) and (2) of Theorem 1, and it satisfies
condition (2) of Theorem 2.  The matrix $XA$ found by Maple 11 does not, however, satisfy condition (3) of
Theorem 1, because neither $(XA)_{211}$ nor $(XA)_{212}$ is 0.  This leads one to ask if there must be an
optimum solution of $LP(A)$, if the objective value of $LP(A)$ is positive, for which all three conditions
of Theorem 1 are satisfied.

\section{The two step method}
We will assume in this section that $A$ has the P-property and that $A$ satisfies the equivalent conditions of
Theorem 2.  We want to show that $LP(A)$ can be computed in strongly polynomial time if the optimal value of $LP(A)$ is a constant and not part of the input.  In addition, an optimal $X$ satisfying all three conditions of Theorem 1 will be found.  First we
recall the matrix $\hat{A}^m$ from the proof of Theorem 1.

\begin{lemma} The matrix $\hat{A}^m$ satisfies the equivalent conditions of Theorem 2.
\end{lemma}

\begin{proof}Suppose  $\overline{X}$ is as in Theorem 1.  Element $(\overline{X}A)_{mmn_m}$ is positive.  Let $E$ be the matrix obtained from the $m \times m$ identity matrix by replacing entry $(j,m)$ by $-(\overline{X}A)_{jmn_m}/(\overline{X}A)_{mmn_m}$ for $j=1,2,\ldots,m-1$.  It is easy to see that if $p^T\overline{X}A > 0$, then $p^TE^{-1}$ will be positive and satisfy
$(p^TE^{-1})(E\overline{X}A)>0$.    Form $\overline{A}^m$ from $E\overline{X}A$ by deleting row $m$ and columns
$(m,1),(m,2),\ldots,(m,n_m)$.
Matrix $\overline{A}^m$ satisfies the equivalent conditions of Theorem 2.  Furthermore, there
exists an $(m-1) \times (m-1)$ matrix $X^m$ so that $\overline{A}^m = X^m\hat{A}^m$, where $\hat{A}^m$ is from the
proof of Theorem 1.
 \end{proof}
\begin{proposition} Row $m$ of $\overline{X}$ can be found by solving a Markov decision problem with matrix $\hat{A}^m$.\end{proposition}
\begin{proof}As in the proof of Theorem 1, define
$c$ to be $-1$ times the vector obtained from row $m$ of $\hat A$ by deleting entries $(m,1),(m,2),\ldots,(m,n_m)$.  Because $\hat{A}^m$ satisfies the equivalent conditions of Theorem 2, there exists a vector $b \in \mathbb{R}^{m-1}$ so that the representative submatrices of $\hat{A}^m$ are the nondegenerate feasible bases of $\hat{A}^mx=b, x \ge 0.$
The vector $v$ that was found in the proof of Theorem 1 can be found as the dual solution to the linear program
minimize $c^Tx$ subject to $\hat{A}^mx=b, x \ge 0$.  Note that one does not need to know $b$ in order to solve this linear program by a pivoting method.  If the entering variable is indexed by the pair $(j,k)$, then the leaving variable is the currently basic variable indexed by $(j,k')$ for some $k' \neq k$.
For $(v,1)$ obtained by appending a 1 to $v$ we obtain $(v,1)\hat{A}$, a vector in the row space of $A$ that is nonpositive on entries not indexed by $(m,1),(m,2),\ldots,(m,n_m)$ and has for each $j = 1,2,\ldots,m-1$ an entry $(j,k)$ that is zero.  It is also positive on entry $(m,n_m)$, so by the sign-preserving property of the row space of $A$ it is positive on all components $(m,1),(m,2),\ldots,(m,n_m)$.  Thus $(v,1)^T\hat{C}^{-1}$ can serve as row $m$ of matrix $\overline{X}$. \end{proof}

We can obtain the other rows of $\overline{X}$ similarly, by solving linear programs.

 \begin{example}   Applied to the previous example, we obtain
   $$\overline{X} = \left[ {\begin{array}{*{20}c}
 1/3 & 1/9 & 1/9 \\
 3/19 & 7/19& 5/38  \\
 4/33 & 1/11 & 10/33
      \end{array} } \right] $$ and
      $$\overline{X}A = \left[ {\begin{array}{*{20}c}
 1 & 1& | & 0 & -5/9 & | & -1/3 & 0 \\
  -9/38 & 0 & |& 45/38 & 1 & | & -3/19 & 0 \\
  0 & -7/33 & |& -2/33 & 0 & | & 29/33& 1
      \end{array} } \right]. $$  Note that the entry in position $(j,j,n_j)$ is 1 for all $j$.
\end{example}

 The second step of our method is to determine an optimum scaling of the rows of the matrix $\overline{X}$.  Let $\overline{A} = \overline{X}A$.  The column sums of $\overline{A}$ are not necessarily between 0 and 1.  We find an optimum scaling vector $x \in \mathbb{R}^m$ by solving the linear program: maximize $d$, subject to $x^T\overline{A} \ge de^n$, $x_jA_{jjk}\le 1$ for all $j = 1,2,\ldots,m$, $k \in [n_j]$.  If $diag(x)$ is the $m \times m$ diagonal matrix with $x$ on the diagonal, we see that $x^T\overline{A} = (e^m)^Tdiag(x)\overline{A}$, so the vector $x$ can be thought of as an optimal scaling of the rows of $\overline{A}$ to bring it into the MDP format.
 We call this linear program the scaling LP.  The matrix $diag(x)$ is feasible for $LP(\overline{A})$, and it has the extra feature that condition (3) of Theorem 1 is satisfied.  It is not immediately obvious, though, that the optimal value of the scaling LP is equal to the optimal value of $LP(A)$.  To prove the equality is our remaining goal.

 \begin{example} Applying the scaling LP to $\overline{A}$ from our previous example yields
  $$diag(x) = \left[ {\begin{array}{*{20}c}
 47/70 & 0 & 0 \\
 0 & 38/45&  0 \\
 0 & 0 & 33/35
      \end{array} } \right],d = 33/70 $$ and
      $$diag(x)\overline{X}A = \left[ {\begin{array}{*{20}c}
 47/70 & 47/70& | & 0 & -47/126 & | & -47/210 & 0 \\
  -1/5 & 0 & |& 1 & 38/45 & | & -2/15 & 0 \\
  0 & -1/5 & |& -2/35 & 0 & | & 29/35& 33/35
      \end{array} } \right]. $$
\end{example}

\begin{proposition} The linear program $LP(A)$ has the same optimal value as the scaling LP.
\end{proposition}
\begin{proof}
 It is clear that the linear programs $LP(A)$ and $LP(\overline{A})$ have the same optimal value, so we want to show that the optimal values of $LP(\overline{A})$ and the scaling LP are the same.  Note that the optimal values of both
 LPs are positive.  For the scaling LP, this implies, by (3) of Proposition 1, that the solution $x$ will be nonnegative.   We will compare the dual linear programs.  The dual to $LP(\overline{A})$ has variables $y_{ijk}$ corresponding to entries of $\overline{A}$ and variables $w_{jk}$ corresponding to columns of $\overline{A}$.

$$  {\begin{array}{*{20}c}
 \text{minimize } \sum_{i=j\in [m],k \in [n_j]} y_{ijk} \text{ subject to}  \\
  \sum_{j\in [m],k \in [n_j]} \overline{A}_{i'jk}y_{ijk} = \sum_{j\in[m],k\in [n_j]}\overline{A}_{i'jk}w_{jk} \text{ for all } i\in [m],i'\in [m] \\
  \sum_{j\in [m],k\in [n_j]} w_{jk} = 1,\\
   y \ge 0, w \ge 0
      \end{array} }.  $$

The scaling LP has variables $y_{jjk}$ corresponding to entries of $\overline{A}$ for which the first two indices are equal, and variables $w_{jk}$ corresponding to columns of $\overline{A}$.

$$  {\begin{array}{*{20}c}
 \text{minimize } \sum_{i=j\in [m],k \in [n_j]} y_{ijk} \text{ subject to}  \\
  \sum_{k \in [n_i]} \overline{A}_{iik}y_{iik} = \sum_{j\in[m],k\in [n_j]}\overline{A}_{ijk}w_{jk} \text{ for all } i\in [m] \\
  \sum_{j\in [m],k\in [n_j]} w_{jk} = 1,\\
   y \ge 0, w \ge 0
      \end{array} }.  $$

Suppose that $\{y^*_{ijk}:i=j\}$ and $\{w^*_{jk}\}$ are an optimal solution to the dual of the scaling LP.  We will
create a solution $(\overline{y},\overline{w})$ to the dual of $LP(\overline{A})$ by solving systems of equations.
We will keep the solution to the scaling LP:  $\overline{y}_{ijk}=y^*_{ijk}$ whenever $i=j$ and $\overline{w}=w^*$.

For each $i = 1,2,\ldots,m$, we will form a linear system to solve for variables $\overline{y}_{ijk}$ for $j \neq i$.  Given $i$, let $K_i$ be a function that assigns to each $j \neq i$ a pair $(j,k)$ for which $\overline{A}_{ijk}=0$.
System $i$ will have an equation for each $i' \neq i$:
$$\sum_{j \neq i}\overline{A}_{i'jK_{i}(j)}\overline{y}_{ijK_{i}(j)}=\sum_{j\in [m],k\in[n_j]}\overline{A}_{i'jk}w^*_{jk} - \sum_{k \in [n_i]}\overline{A}_{i'ik}y^*_{iik},$$

It is easy to verify that the right hand side of this system is nonnegative. The coefficient matrix of the system is an
$(m-1) \times (m-1)$ $K$-matrix, so the system will have a nonnegative solution when the right hand side is nonnegative. The systems for different $i$ have disjoint sets of variables, so they are solved independently.  All variables $\overline{y}_{ijk}$ not specified by the scaling LP or by the systems are set to 0.  The resulting $(\overline{y},\overline{w})$ is a feasible solution to the dual LP of $LP(\overline{A})$.  The objective value of this dual solution is the same as that of the scaling LP, because the variables in the two objective functions are the same.
\end{proof}

\section{Complexity}
Now we assume that $(\overline{X},d)$ solves $LP(A)$ and that $d>0$.  We will assume that the matrix $\overline{X}$ satisfies conditions $(1)-(3)$ of Theorem 1.  We have $(e^m)^T\overline{X}A \ge d(e^n)^T$.
\begin{proposition} Row $m$ of $\overline{X}$ is found by a strongly polynomial algorithm, solving a Markov decision problem with matrix $\hat{A}^m$, if $d$ is considered to be a constant.\end{proposition}
\begin{proof}Let $E$ be the matrix obtained from the $m \times m$ identity matrix by replacing entry $(j,m)$ by $-(\overline{X}A)_{jmn_m}/(\overline{X}A)_{mmn_m}$ for $j=1,2,\ldots,m-1$.  Form $\overline{A}^m$ from $E\overline{X}A$ by deleting row $m$ and columns $(m,1),(m,2),\ldots,(m,n_m)$. Then for any $(j,k)$ with $j \in [m-1]$ and $k \in [n_j]$, we have $d \le (e^m)^T\overline{A}_{\cdot jk} = (e^{m-1})^T\overline{A}^m _{\cdot jk} + (1-\sum_{j=1}^{m-1}(\overline{X}A)_ {jmn_m}/(\overline{X}A)_{mmn_m})(\overline{X}A)_{mjk}$.  This last term is nonpositive, so $(e^{m-1})^T\overline{A}^m _{\cdot jk}\ge d$.  It follows from \cite{Ye} that the linear program to compute row $m$ of $\overline{X}$ (without scaling) can be solved in at most $\frac{m(n-m)}{d}\cdot log(\frac{m^2}{d})$ pivots using Dantzig's rule, or with similar complexity using the results of $\cite{HMZ}$.
\end{proof}

\begin{proposition} The scaling LP can be solved by applying the strongly polynomial algorithm of \cite{HMZ}, again assuming $d$ is a constant, followed by solving a one-variable linear program.
\end{proposition}

\begin{proof}Here we first find a vector $v$ satisfying $v^T\overline{A} \ge (e^n)^T$, such that
for every $j \in [m]$ there exists a $k \in [n_j]$ such that component $(j,k)$ of  $(e^n)^T-v^T\overline{A}$ is zero.
This vector can be found by the strongly polynomial algorithm of \cite{HMZ}.  Then we find the largest $d$ so that
$dv_i\overline{A}_{iik} \le 1$ for all $i \in [m]$, $k \in [n_i]$.  Let $(i,i,k)$ index an entry of $d(diag(v))\overline{A}$ that is equal to 1.  Let $C$ be a representative submatrix of $d(diag(v))\overline{A}$ satisfying $(e^m)^TC = d(e^m)^T$.  We build a dual solution $(y,w)$ to the scaling LP as follows.  We let $y_{iik} = d$,  and let the components of $w$ corresponding to the columns of $C$ satisfy $Cw=de_i$.  Let all other components of $y$ and $w$ be 0.  Then the system $Cw=de_i$ guarantees that the constraints $$\sum_{k \in [n_i]} \overline{A}_{iik}y_{iik} = \sum_{j\in[m],k\in [n_j]}\overline{A}_{ijk}w_{jk} \text{ for all } i\in [m] $$ will be satisfied.  Premultiplying both sides of $Cw=de_i$ by $(e^m)^T$ gives $(e^m)^TCw = d(e^m)^Te_i$, or $(e^m)^Tw = 1$.  Thus $(y,w)$ is a feasible solution to the dual of the scaling LP, which has the same value as the solution $dv$ of the primal LP.
\end{proof}

\begin{theorem} If the optimal value of the linear program $LP(A)$ is positive, and this optimal value is considered a constant that is not part of the input data, then there is a strongly polynomial algorithm to solve $LP(A)$.
\end{theorem}

It is interesting that the two step method described here serves two seemingly independent purposes, to satisfy condition (3) of Theorem 1 and to be efficient.  There does not appear to be reason to assume that the straightforward application of, say, the simplex method to $LP(A)$ would be strongly polynomial.  The solution $diag(x)\overline{X}$ found to $LP(A)$ also has the property that for each $j \in [m]$ there is an index $(j,k)$ for which $x^T\overline{X}A_{jk} = d$, where $d$ is the optimal value of the LP.

\end{document}